\documentclass[reqno, 12pt]{amsart}
\usepackage[utf8]{inputenc}
\usepackage[T1]{fontenc}

\usepackage[normalem]{ulem}
\usepackage{amsmath}
\usepackage{amssymb}
\usepackage{amsfonts}
\usepackage{graphicx}
\usepackage{amsthm}
\usepackage{enumerate}
\usepackage{lscape}
\usepackage{dsfont}
\usepackage{color}
\usepackage{mathtools}

\usepackage{hyperref}

\hypersetup{
	colorlinks=true,
	linkcolor=blue,
	citecolor=blue,
	urlcolor=blue
}

\usepackage{cleveref}

\usepackage{setspace}

\newcommand{\R}{\mathds{R}}

\newcommand{\CP}{\mathds{C}\mathrm{P}}
\newcommand{\N}{\mathds{N}}

\newcommand{\C}{\mathds{C}}

\newcommand{\K}{K\"{a}hler}

\def\b{\beta}

\def\b1{{\rm id}}

\newtheorem{theor}{Theorem}[section]

\newtheorem{lem}[theor]{Lemma}
\newtheorem{cor}[theor]{Corollary}

\newtheorem{example}{Example}
\newtheorem{rmk}{Remark}

\crefname{proposition}{proposition}{propositions}
\Crefname{proposition}{Proposition}{Propositions}

\begin{document}
	
	\title[Radial Projectively Induced Canonical Metrics]{Radial Projectively Induced Canonical K\"ahler Metrics: Rigidity and Classification}
	
	\author[C. Arezzo]{Claudio Arezzo}
	\address{(Claudio Arezzo) Mathematics Section\\ International Centre for Theoretical Physics
	}
	\email{arezzo@ictp.it}
	
	\author[A. Loi]{Andrea Loi}
	\address{(Andrea Loi) Dipartimento di Matematica \\
		Universit\`a di Cagliari, Via Ospedale 72, 09124  (Italy)}
	\email{loi@unica.it}
	
	\author[G. Placini]{Giovanni Placini}
	\address{(Giovanni Placini) Dipartimento di Matematica \\
		Universit\`a di Cagliari, Via Ospedale 72, 09124  (Italy)}
	\email{giovanni.placini@unica.it}

	\author[M. Zedda]{Michela Zedda}
	\address{(Michela Zedda) Dipartimento di Scienze Matematiche, Fisiche e Informatiche \\
		Universit\`a di Parma (Italy)}
	\email{michela.zedda@unipr.it}

	\thanks{
		This paper has been partially supported by the group GNSAGA of INdAM. 
		The second and third authors have been supported  by ProBiki of Fondazione di Sardegna.
		The fourth author has also been supported by the project Prin 2022 – Real and Complex Manifolds: Geometry and Holomorphic Dynamics-Italy.
	}

	\subjclass[2020]{53C55, 32Q20, 32Q57}
\makeatletter
\renewcommand{\subjclassname}{2020 Mathematics Subject Classification}
\makeatother
	\keywords{K\"ahler--Einstein metric; constant scalar curvature K\"ahler metric; Calabi's rigidity theorem; K\"ahler immersion}

\begin{abstract}
We study radial K\"ahler  metrics on domains of $\mathds{C}^n$, $n\geq 2$, admitting a K\"ahler  immersion into a finite- or infinite-dimensional complex projective space. We classify those with constant non-negative scalar curvature: up to a linear change of coordinates, they are positive integer multiples of the Fubini--Study metric, the flat metric, or, in complex dimension two, generalized Burns--Simanca metrics. We also prove that every radial projectively induced K\"ahler--Einstein metric has constant holomorphic sectional curvature and is therefore a Fubini--Study, flat, or complex hyperbolic metric. Finally, we show that a radial infinitely projectively induced extremal K\"ahler metric has unbounded maximal radial domain if and only if it is scalar-flat.
\end{abstract}

\maketitle
	
	\tableofcontents  

\section{Introduction and statement of the results}\label{sec:intro}

A  K\"ahler metric $g$ on a complex manifold $M$ is said to be \emph{projectively induced} if $(M,g)$ admits a  K\"ahler immersion, namely a holomorphic and isometric immersion, into the complex projective space $(\mathds{C}{\rm P}^{N},g_{FS})$, $N\leq\infty$, endowed with the Fubini--Study metric. We refer to \cite{LoiZeddaBook} for an account of the subject, whose foundations go back to the seminal work of Calabi \cite{calabi}. A  K\"ahler immersion is said to be \emph{full} if its image is not contained in any proper totally geodesic complex projective subspace of the ambient projective space. Equivalently, the projective span of its image is the whole ambient projective space. We say that $g$ is \emph{finitely projectively induced} or \emph{infinitely projectively induced} according to whether it admits a full  K\"ahler immersion into $\mathds C{\rm P}^{N}$ with $N<\infty$ or into $\mathds C{\rm P}^{\infty}$, respectively.

A classical problem is to determine which canonical  K\"ahler metrics, such as  K\"ahler-Einstein (KE), constant scalar curvature  K\"ahler (cscK), or, more generally, extremal metrics in the sense of Calabi \cite{calextrem}, are projectively induced. The finite-dimensional setting displays strong rigidity phenomena. Recent rigidity results for holomorphic isometries into finite-dimensional projective spaces can be found in \cite{ArezzoLiLoi}. Among the classical results, Chern \cite{chern} and Tsukada \cite{tsukada} classified finitely projectively induced KE metrics in low codimension, while Hulin proved that the Einstein constant of a compact projectively induced KE manifold is positive \cite{HU}; see also \cite{kob,KON} for results concerning the cscK condition in low codimension. In the non-positive curvature setting, Umehara \cite{umehara} obtained rigidity results for Einstein  K\"ahler submanifolds of complex Euclidean and hyperbolic spaces. 
Conjecturally, the only finitely projectively induced KE metrics are open subsets of homogeneous  K\"ahler manifolds.

The infinite-dimensional setting is considerably more flexible. By Calabi's criterion,  the flat metric $g_0$ on $\mathds{C}^n$ and every positive multiple  of the complex hyperbolic metric $g_{hyp}$ is infinitely projectively induced, whereas $c g_{FS}$ is projectively induced if and only if $c\in\mathds{Z}^{+}$; see \cite{calabi,LoiZeddaBook}. Moreover, non-homogeneous KE submanifolds of $\mathds{C}{\rm P}^{\infty}$ were constructed in \cite{LoiZedda11,VARI}. Nevertheless, several rigidity phenomena arise under additional geometric assumptions. For instance, Loi, Salis and Zuddas conjectured that every Ricci-flat projectively induced  K\"ahler metric is flat \cite{LSZ}; this conjecture has been verified for several distinguished families, including the Calabi and Stenzel metrics \cite{LOIZEZU,ZStenzel}. One cannot hope for infinite-dimensional analogues of \cite{HU,kob,KON}, as there are examples of non-homogeneous complete (negative) KE metrics immersed in $\CP^\infty$ \cite{LoiZedda11,VARI}, or even in any infinite-dimensional complex space form \cite{CaibarLoi}.

The starting point of the present paper is \cite{LSZext}, where Loi, Salis and Zuddas proved several results on extremal radial metrics admitting  a  K\"ahler immersion into a complex space form. Recall that a  K\"ahler metric on a domain of $\mathds{C}^{n}$ is \emph{radial} if its  K\"ahler form can be written as
\begin{equation*}
\omega=\frac{i}{2}\partial\bar\partial f(r),
\qquad
r=|z|^2.
\end{equation*}
We denote by $(r_{\mathrm{inf}},r_{\mathrm{sup}})$ the maximal interval of definition of the radial potential $f$. For $k\in\mathds Z^+$, we call \emph{generalized Burns--Simanca metric} the radial  K\"ahler metric $g_{k,BS}$ on $\mathds C^2\setminus\{0\}$ whose  K\"ahler form is
$$
\omega_{k,BS}=\frac{i}{2}\partial\bar\partial\bigl(k\log|z|^2+|z|^2\bigr).
$$
The case $k=1$ is the classical Burns--Simanca metric; see \cite{Simanca}.

The main result of this paper is the following.

\begin{theor}\label{thmNew}
Let $g$ be a radial projectively induced extremal  K\"ahler metric on a domain of $\mathds{C}^{n}$, $n\geq2$. Then the following hold:
\begin{enumerate}[1)]
\item $r_{\mathrm{sup}}=+\infty$ if and only if $g$ is a cscK metric with non-negative scalar curvature. More precisely, if $g$ is infinitely projectively induced, then
$r_{\mathrm{sup}}=+\infty$ if and only if 
$g$ is scalar-flat.\label{PartExtremal}
\item
If $g$ is a cscK metric with non-negative scalar curvature, then, up to a linear change of coordinates, one of the following holds: \label{PartcscK}
	\begin{enumerate}[(2a)]
		\item $g$ is a positive integer multiple of the Fubini--Study metric; \label{CasoA}
		\item $g$ is the flat metric; \label{CasoB}
		\item $n=2$ and $g$ is a generalized Burns--Simanca metric.
		\label{CasoC}
	\end{enumerate}

	\item If $g$ is KE, then it has constant holomorphic sectional curvature. \label{PartKE}
\end{enumerate}

\end{theor}

In the finitely projectively induced case the behaviour is much more rigid: \cite[Theorem~1.1]{LSZext} forces $g$ to have positive constant holomorphic sectional curvature, and hence, up to a linear change of coordinates, $g$ is a positive integer multiple of the Fubini--Study metric. In particular, $g$ is cscK with positive scalar curvature and its maximal radial domain satisfies $r_{\mathrm{sup}}=+\infty$.

The condition $r_{\mathrm{sup}}=+\infty$ in part \eqref{PartExtremal} is essential. Indeed, there exist radial extremal  K\"ahler metrics defined on an annulus and immersed in $\CP^\infty$ which do not have constant scalar curvature; see, for instance, \cite[Example~1]{LSZext}. Conversely, the cscK condition alone does not imply $r_{\mathrm{sup}}=+\infty$: Section \ref{sec:examples} contains infinitely projectively induced radial cscK metrics with negative scalar curvature and bounded maximal radial domain.

Part \eqref{PartcscK} was proved in \cite[Theorem~1.3]{LSZext} under an additional technical assumption on the behaviour of the  K\"ahler potential near the boundary of its domain of definition; see \cite[Definition~1]{LSZext}. While we remove this hypothesis, the assumption on the sign of the scalar curvature is essential. Indeed, Section \ref{sec:examples} provides two families of infinitely projectively induced radial cscK metrics with negative scalar curvature which do not have constant holomorphic sectional curvature. As shown by part \eqref{PartKE}, the  K\"ahler-Einstein condition, however, restores rigidity.
The same conclusion was obtained in \cite[Theorem~1.4]{LSZext} under a further hypothesis on the stability of the immersion. We remove this assumption and, therefore, give a positive answer to \cite[Conjecture~2]{LSZext}.

{\textbf{Organization of the paper:}} The starting point of the proof of \Cref{thmNew} is to show that a projectively induced radial  K\"ahler metric is always defined on a punctured disc; see \Cref{lemmaimp}. This is proved in Section \ref{sec:radial}, where we establish the structural properties of radial metrics induced by a full  K\"ahler immersion into $\mathds{C}{\rm P}^{\infty}$. In Section \ref{sec:classification} we describe the behaviour of the  K\"ahler potential of an extremal radial  K\"ahler metric and its domain of definition, and prove parts \eqref{PartExtremal} and \eqref{PartcscK} of \Cref{thmNew}. Finally, Section \ref{sec:examples} contains examples of infinitely projectively induced radial cscK metrics with negative scalar curvature and the proof of \eqref{PartKE} of \Cref{thmNew}.
            
            \section{Radial K\"ahler metrics induced by a full K\"ahler immersion in $\CP^\infty$}\label{sec:radial}
Let $g$ be a radial  \K\ metric on a complex manifold $M$, equipped with complex coordinates $z_1, \dots ,z_n$. Assume $n\geq 2$.
	Let  $\omega =\frac{i}{2}\partial\bar\partial f(r)$ denote its associated \K\ form
	where $f:(r_{\inf}, r_{\sup})\rightarrow \R$, $r =|z|^2= |z_1|^2 + \cdots + |z_n|^2$, $0\leq r_{\inf}<r<r_{sup}$
	and $(r_{\inf}, r_{\sup})$ is the maximal domain where the radial potential $f$ is defined.
    
	In other words  the metric is defined on  
	the set
	$$D_{r_{\inf}, r_{\sup}}=\{z\in \C^n \ | \ r_{\inf}<|z|^2<r_{\sup}\}$$

	It is not hard to see that 
	the matrix of the metric $g$ and its inverse read  as
	\begin{equation}\label{metric}
		g_{i\bar j}=f''(r) \bar z_i z_j+f'(r)\delta_{ij}, \ \ \  g^{i\bar j}=\frac{\delta_{ij}}{f'(r)} -\frac{f''(r)}{f'(r) (rf'(r))'}\bar z_j z_i.
	\end{equation}

	Set 
	\begin{equation}\label{y(r)}
		y(r):=rf'(r).
	\end{equation}
	The fact that $g$ is a metric is equivalent to 
	$y(r)>0$ and $ry'(r)>0$, $\forall r\in (r_{\inf}, r_{\sup})$.

            In this section we describe the behaviour of radial K\"ahler metrics that are induced by a full K\"ahler immersion in $\CP^\infty$.            
           Theorem \ref{lemmaimp} below shows that radial K\"ahler metrics induced by a full K\"ahler immersion in $\CP^\infty$ are actually defined on a punctured disk $D_{0, r_{\sup}}$ and the exponential of their K\"ahler potential $F(r):=e^{f(r)}$ admits an expansion in $r$ on the whole ball.

Let us start with the following lemma.

	\begin{lem}\label{lemradiallogk}
		Let 
		\[
		F(r)=\sum_{j=0}^\infty a_j\,r^j,\quad a_j\ge0,\quad 0\leq r<r_{\sup},
		\]
		be a real analytic radial function on the open  ball $B_{\sqrt{r_{\sup}}}(0)$ with non-negative coefficients $a_i$ for all $i\in \N$. Assume there exist at least two distinct indices \(j\neq k\) with \(a_j>0\) and \(a_k>0\).  Then the \((1,1)\) form
		\[
		\frac{i}{2}\,\partial\bar\partial\log F(|z|^2)
		\]
		is a K\"ahler form on the punctured ball $D_{0, r_{\sup}}$.
	\end{lem}
	
	\begin{proof}
		Set 
		\[
		f(r)=\log  F(r),
		\]
		The radial \((1,1)\)-form \(\tfrac{i}{2}\partial\bar\partial f(|z|^2)\) is \K\  if and only if for all \(r>0\) one has
		\[
		f'(r)>0
		\quad\text{and}\quad
		f'(r)+r\,f''(r)>0.
		\]
		
		\medskip\noindent
		\emph{Step 1:} \(f'(r)>0\).
		\[
		f'(r)
		=\frac{F'(r)}{F(r)}
		=\frac{\sum_{j=1}^\infty j\,a_j\,r^{j-1}}{\sum_{j=0}^\infty a_j\,r^j}
		>0,
		\]
		since at least one \(a_j>0\) for \(j\ge1\) and the denominator is positive.
		
		\medskip\noindent
		\emph{Step 2:} \(f'(r)+r\,f''(r)>0\).
		
		A direct calculation gives
		\[
		f'(r)+r\,f''(r)
		=\frac{F\,F' + r\bigl(F\,F''-(F')^2\bigr)}{F^2}
		=\frac{1}{2F^2}\sum_{j,k=0}^\infty (k-j)^2\,a_j\,a_k\,r^{\,j+k-1}.
		\]

		Each summand \((k-j)^2a_j a_k r^{j+k-1}\ge0\), and is strictly positive whenever \(j\neq k\) with \(a_j,a_k>0\). By assumption there exist two such indices, so the entire sum is positive.
	\end{proof}

	\begin{theor}\label{lemmaimp}
		Assume $g$  is infinitely projectively induced and radial on the annulus $D_{r_{\inf}, r_{\sup}}$, i.e. 
		$\omega =\frac{i}{2}\partial\bar\partial f(r)$ as above. Then $r_{\inf}=0$. 
		Moreover, there exists
		a sequence $a_j$, $j=0, \dots$ of non-negative real numbers with  infinitely many of them strictly positive 
		such that
		\begin{equation}\label{efr}
			F(r)=e^{f(r)} \;=\; \sum_{j=0}^{+\infty} a_jr^j, \  0\leq r< r_{\sup}.
		\end{equation}
	\end{theor}
	
	\begin{proof}
		Fix $r_0$ with $r_{\inf}\leq r_0<r_{\sup}$.
		Since $g$ is infinitely projectively induced for any  $p$ of the annulus $D_{r_0, r_{\sup}}$
		we can find a sequence $f_{p,k}$, for $k=1,2, \dots$, of holomorphic functions on a neighbourhood $U$ of $p$ such that 
		\[
		e^{f(|z|^2)}=e^{f(r)} \;=\; \sum_{k=1}^\infty \bigl|f_{p,k}(z)\bigr|^2.
		\]
		Since $D_{r_0, r_{\sup}}$ is simply-connected ($n\geq 2$), by Calabi's extension theorem
		we can find a sequence $f_k$, $k=1,2 \dots$ of holomorphic functions on $D_{r_0, r_{\sup}}$ extending $f_{ p,k}$ such that 
		\begin{equation}\label{efr2}
			e^{f(|z|^2)} \;=\; \sum_{k=1}^\infty \bigl|f_k(z)\bigr|^2
		\end{equation}
		and hence
		\begin{equation}\label{omegaloc}
			\omega=\frac{i}{2}\partial\bar\partial f(|z|^2)=\frac{i}{2}\partial\bar\partial\log \sum_{k=1}^\infty \bigl|f_k(z)\bigr|^2.
		\end{equation}
		By Hartogs' extension theorem we can extend each $f_k$ to a holomorphic function $\tilde f_k$
		on the open ball $B_{\sqrt{r_{\sup}}}(0)$. Consequently also $e^{f(r)}$ can be extended to a radial function 
		$F(r)$ on $B_{\sqrt{r_{\sup}}}(0)$ such that 
		\[
		F(|z|^2)=\; \sum_{k=1}^\infty \bigl|\tilde f_k(z)\bigr|^2.
		\]

		Write each \(\tilde f_k\) in its Taylor series around \(0\):
		\[
		\tilde f_k(z)
		= \sum_{|\alpha|\ge0} c_{k,\alpha}\,z^\alpha,
		\qquad
		z^\alpha = z_1^{\alpha_1}\cdots z_n^{\alpha_n}.
		\]
		Then
		\[
		\bigl|\tilde f_k(z)\bigr|^2
		= \sum_{\alpha,\beta}
		c_{k,\alpha}\,\overline{c_{k,\beta}}\;
		z^\alpha\,\overline{z}{}^{\beta},
		\]
		and hence
		\[
		\sum_{k=0}^\infty \bigl|\tilde f_k(z)\bigr|^2
		= \sum_{\alpha,\beta}
		\Bigl(\sum_k c_{k,\alpha}\,\overline{c_{k,\beta}}\Bigr)
		z^\alpha\,\overline{z}{}^{\beta}.
		\]
		But radial symmetry of the left hand side forces all mixed terms with \(\alpha\neq\beta\) to vanish and the coefficients of the terms with $\vert\alpha\vert=\vert\beta\vert$ to be proportional to the multinomial coefficients $\frac{\vert\alpha\vert!}{\alpha_1!\cdots\alpha_n!}$ . Thus writing $J=(j,0,\ldots,0)$ for $j\in \N$ so that $\vert J\vert=j$ we can rewrite the expression above as
		\begin{equation}\label{tildeF}
			F(r)
			= \sum_{j=0}^\infty
			\Bigl(\sum_{k} |c_{k,J}|^2\Bigr)\,
			|z|^{2j}
			= \sum_{j=0}^\infty a_j\,|z|^{2j}=\sum_{j=0}^{+\infty} a_jr^j, \  0\leq r< r_{\sup},
		\end{equation}
		where
		\[
		a_j = \sum_{k} |c_{k,J}|^2 \;\ge\;0.
		\]

Since $f(r)$ is a K\"ahler potential for $r\in (r_0,  r_{\sup})$, then \begin{equation}\label{forF}
		F(r)=e^{f(r)}=\sum_{j=0}^{+\infty} a_jr^j,
	\end{equation} does not vanish identically for $r\in (r_0,  r_{\sup})$ and thus at least one of the $a_j$'s in \eqref{tildeF} is positive. 
		Thus the holomorphic map
		$$\varphi :D_{0, r_{\sup}}\rightarrow \C{\rm P}^{\infty}, \ z\mapsto \left[\dots ,\ \sqrt{\frac{j!\,a_j}{\alpha!}}\,z^\alpha \, ,\dots\right]_{j\geq 0,\, |\alpha|=j}$$
		(where we used the multi-index notation
$\alpha!=\alpha_1!\cdots\alpha_n!$,
$z^\alpha=z_1^{\alpha_1}\cdots z_n^{\alpha_n}$)
is well-defined and, by \eqref{efr2}, \eqref{omegaloc} and \eqref{tildeF} its restriction   $\varphi_{|D_{r_0, r_{\sup}}}$ is a  holomorphic immersion inducing $\omega$, i.e.
		$$\varphi_{|D_{r_0, r_{\sup}}}^*\omega_{FS}=\omega_{|D_{r_0, r_{\sup}}}.$$
		Thus, by Calabi's rigidity Theorem (see \cite{calabi,LoiZeddaBook}) and since the immersion is full, infinitely many of the $a_j$'s are strictly positive.
		By  Lemma \ref{lemradiallogk} we also get $\varphi^*\omega_{FS}=\frac{i}{2}\partial\bar\partial\log F(|z|^2)$
		is a \K\ form on the punctured ball $D_{0, r_{\sup}}$ (i.e. $\varphi$ is a holomorphic immersion) extending $\omega_{|D_{r_0, r_{\sup}}}=\frac{i}{2}\partial\bar\partial f(|z|^2)_{|D_{r_0, r_{\sup}}}$. We deduce that  $f(r)=\log F(r)$ is defined for $r\in (0, r_{\sup})$, i.e. $r_{\inf}=0$. Hence \eqref{efr} is proved.
	\end{proof}

			We conclude this section providing an expansion of the function $y(r)=rf'(r)$ near $r=0$ in terms of the first non-vanishing terms in the expansion of $F$ obtained in Theorem \ref{lemmaimp}, which will be used in the next section.

	\begin{cor}\label{lem:y_expansion}
    Let $g$ be a radial K\"ahler metric induced by a full K\"ahler immersion in $\mathds C{\rm P}^\infty$. Then as \(r\to0\) the function $y(r)$ associated to $g$ admits the expansion:
    \begin{equation}\label{eqexpansiony}
		y(r)=k + K\,r^l + o(r^l),
		\end{equation}
        where $k\in\mathds Z$, $k\geq 0$, and $K>0$.
	\end{cor}

	\begin{proof}
		Since the metric is infinitely projectively induced, by Theorem \ref{lemmaimp} we can write the following expansion for the function $F(r)=e^{f(r)}$:
        \[
		F(r)=\sum_{j=0}^\infty a_jr^j,
		\]
        where the $a_j$'s are all non-negative and infinitely many of them are positive. Denote by $k$ and $l+k$ the indices of the first two non-vanishing coefficients, that is:
        \[
		F(r)=a_k\,r^k + a_{k+l}\,r^{k+l} + o(r^{k+l}).
		\]
        By \eqref{y(r)}, $y(r)=rf'(r)$, and thus 
		\[
		\begin{aligned}
			y(r)
			&=\frac{r\,F'(r)}{F(r)}
			=\frac{k\,a_k + (k+l)\,a_{k+l}\,r^l + o(r^l)}
			{a_k + a_{k+l}\,r^l + o(r^l)}=\frac{k + (k+l)\frac{a_{k+l}}{a_k}\,r^l + o(r^l)}
			{1 + \frac{a_{k+l}}{a_k}\,r^l + o(r^l)}.
		\end{aligned}
		\]   
		Setting $u=  (k+l)\frac{a_{k+l}}{a_k}\,r^l + o(r^l)$ and $v=\frac{a_{k+l}}{a_k}\,r^l + o(r^l)$ by the standard expansion \(\frac{k+u}{1+v}=k+(u-kv)+o(r^l)\), we get
		\[
		y(r)
		= k+(u-kv)+o(r^l).
		\]
		Noting also that:
		\[
		u-kv = l\frac{a_{k+l}}{a_k}\,r^l + o(r^l),
		\]
		we conclude
		\[
		y(r)=k + K\,r^l + o(r^l),
		\]
        for $K = l\,\frac{a_{k+l}}{a_k}>0$.
	\end{proof}

	\begin{rmk}\label{rem1}\rm 
Observe that the condition that the metric is projectively induced cannot be dropped, not even if the metric is real analytic and $e^{f(r)}$ extends to the whole ball. In fact, consider the classical Burns--Simanca metric $g_{BS}=g_{1,BS}$, whose K\"ahler potential reads $f(r)=r+\log(r)$ defined on $\mathds C^2\setminus \{0\}$. A multiple $c g_{BS}$ of the Burns--Simanca metric (see e.g. \cite{simreg}) is infinitely projectively induced if and only if $c\in\mathds Z^+$. For non integral values of $c$, $c g_{BS}$ is real analytic and the exponential of its potential reads:
$$
F(r)=e^{cf(r)}=r^ce^{cr}=r^c \sum_ {j = 0}^\infty \frac{c^jr^j}{j!},
$$
  that cannot be written in the form $F(r)=\sum_{j=0}^\infty a_jr^j$.
    \end{rmk}

	\section{Radial projectively induced extremal metrics and the cscK classification}\label{sec:classification}
	We begin with a preliminary analysis of radial extremal Kähler metrics, which will also yield part \eqref{PartExtremal} of \Cref{thmNew}. We first recall the following standard characterization; see, for instance, \cite{Abreu,XD,LSZext}.
	
	\begin{theor}\label{lemmasimple}
		A radial \K\ metric  $g$ is extremal if and only if 
		\begin{equation}\label{ypsi}
			ry' = y - \frac{A}{y^{n-1}} - \frac{B}{y^{n-2}} - C y^2 - D y^3. 
		\end{equation}
		for some $A,B,C,D \in \R$.
		Moreover, the following facts hold true:
		\begin{itemize}
			\item [(a)]
			$g$ is a cscK metric\footnote{with constant scalar curvature equal to  $Cn(n+1)$.} if and only if  $D=0$ and the sign of 
			the scalar curvature $S$ is equal to the sign of $C$;
			\item [(b)]
			$g$ is a KE metric with Einstein constant $\lambda$ if and only if   $B=D=0$  and $C = \frac{\lambda}{2(n+1)}$;
			\item [(c)]
			$g$ has constant holomorphic sectional curvature 
			if and only if  $A=B=D=0$.
		\end{itemize}
	\end{theor}

	The following lemma gives the boundary behaviour of the function \(y(r)\) for radial extremal K\"ahler metrics which are infinitely projectively induced.
		\begin{lem}\label{prelemmagood}
		Let $g$ be a radial extremal K\"ahler metric. Assume that $g$ is infinitely projectively induced. Then $y(r)\rightarrow +\infty$ as $r\rightarrow r_{\rm sup}^-$.
			\end{lem}
\begin{proof}
By Theorem \ref{lemmaimp} we have \(r_{\inf}=0\) and
\begin{equation}\label{altraF}
F(r)=e^{f(r)}=\sum_{j=0}^{+\infty}a_j r^j,\qquad 0\leq r<r_{\sup},
\end{equation}
where \(a_j\geq 0\) and infinitely many of the \(a_j\)'s are strictly positive.

Since \(g\) is extremal, by Theorem \ref{lemmasimple} the function \(y(r)=rf'(r)\) satisfies the ODE \eqref{ypsi}.
Moreover, $y$ is strictly increasing, since $ry'(r)>0$ on \((0,r_{\sup})\). 

 Suppose first that \(r_{\sup}<+\infty\). Since \(y\) is positive and increasing, the limit
\[
L:=\lim_{r\to r_{\sup}^{-}}y(r)
\]
exists in \((0,+\infty]\). Assume by contradiction that \(L<+\infty\). Set
\[
s=\log r,\qquad S=\log r_{\sup}.
\]
Then \(y\), as a function of \(s\), satisfies the analytic ODE
\[
\frac{dy}{ds}=G(y),
\qquad
G(Y):=Y-\frac{A}{Y^{n-1}}-\frac{B}{Y^{n-2}}-CY^2-DY^3 .
\]
Since \(L>0\), the function \(G\) is analytic in a neighbourhood of \(L\), because its only
possible singularity is at $Y=0$. By the standard
continuation theorem for analytic ODEs, the solution $y(s)$
extends analytically past $s=S$. Moreover, since \(G(y)>0\) for \(s<S\), we have
\(G(L)\geq 0\). If $G(L)=0$, then the constant function $y\equiv L$ is a local solution with initial value $L$ at $s=S$. By uniqueness for ODEs, the
extended solution would be identically equal to $L$ near $s=S$, contradicting the fact
that $y$ is strictly increasing on the left of $S$. Thus, $G(L)>0$.

Therefore, after possibly shrinking the neighbourhood of \(S\), the extension, that we still denote by $y$, still satisfies $y(s)>0$ and $\frac{dy}{ds}>0$, or equivalently, in the variable $r$, $y(r)>0$, $ry'(r)>0$. It follows that we can extend also the real analytic radial K\"ahler potential $f(r)$ beyond $r_{\sup}$, by setting for a fixed $r_0\in(0,r_{\sup})$ 
$$
\widetilde f(r)=f(r_0)+\int_{r_0}^{r}\frac{ y(t)}{t}\,dt .
$$
This contradicts the maximality of \((0,r_{\sup})\). We have proved so far that if
\(r_{\sup}<+\infty\), then
\[
\lim_{r\to r_{\sup}^{-}}y(r)=+\infty.
\]

Suppose now that \(r_{\sup}=+\infty\) and fix
\(M>0\). Since infinitely many of the coefficients \(a_j\) in \eqref{altraF} are strictly positive, we can
choose \(J>M\) such that \(a_J>0\). Then
\[
F(r)=\sum_{j=0}^{J-1}a_j r^j+\sum_{j=J}^{+\infty}a_j r^j,
\]
and
\[
\sum_{j=0}^{J-1}a_j r^j=O(r^{J-1}),
\qquad
\sum_{j=J}^{+\infty}a_j r^j\geq a_J r^J.
\]
Hence
\[
\frac{\sum_{j=J}^{+\infty}a_j r^j}{\sum_{j=0}^{+\infty}a_j r^j}
\longrightarrow 1
\qquad\text{as } r\to+\infty.
\]
Therefore
\[
y(r)
=
\frac{\sum_{j=1}^{+\infty}j a_j r^j}{\sum_{j=0}^{+\infty}a_j r^j}
\geq
J
\frac{\sum_{j=J}^{+\infty}a_j r^j}{\sum_{j=0}^{+\infty}a_j r^j}.
\]
It follows that
\[
\liminf_{r\to+\infty} y(r)\geq J>M.
\]
Since \(M>0\) is arbitrary, we get \(y(r)\to+\infty\) as \(r\to+\infty\).
\end{proof}

	We are now ready to prove \eqref{PartExtremal} of \Cref{thmNew},

\begin{proof}[Proof of \eqref{PartExtremal} of \Cref{thmNew}]
If $g$ is finitely projectively induced, then \cite[Theorem~1.1]{LSZext} implies that $g$ has positive constant holomorphic sectional curvature. Hence, up to a linear change of coordinates, $g$ is a positive integer multiple of the Fubini--Study metric. In particular, $g$ is cscK with positive scalar curvature and $r_{\mathrm{sup}}=+\infty$.

Assume therefore that $g$ is infinitely projectively induced. Suppose first that
$r_{\mathrm{sup}}=+\infty$.
By Lemma \ref{prelemmagood},
\[
y(r)\longrightarrow+\infty
\qquad\text{as }r\to+\infty.
\]
Since $g$ is radial and extremal, Theorem \ref{lemmasimple} gives
\[
ry'
=
y-\frac{A}{y^{n-1}}-\frac{B}{y^{n-2}}-Cy^2-Dy^3
\]
for some $A,B,C,D\in\R$.
Since $ry'(r)>0$, the extremal equation implies $D\leq0$. Assume by contradiction that $D<0$. Then, for $y$ sufficiently large, there exists $\delta>0$ such that
\[
ry'\geq\delta y^3.
\]
Since $y$ is strictly increasing, we may invert it and write
\[
\frac{d}{dy}\log r
=
\frac{1}{ry'}
\leq
\frac{1}{\delta y^3}.
\]
Integrating from a sufficiently large $Y$ to $y$, we obtain
\[
\log r(y)-\log r(Y)
\leq
\frac{1}{2\delta}
\left(\frac{1}{Y^2}-\frac{1}{y^2}\right).
\]
Thus $r(y)$ remains bounded as $y\to+\infty$, contradicting $r_{\mathrm{sup}}=+\infty$. Hence $D=0$, and therefore $g$ is cscK.
The equation now becomes
\[
ry'
=
y-\frac{A}{y^{n-1}}-\frac{B}{y^{n-2}}-Cy^2.
\]
Since $ry'(r)>0$ and $y(r)\to+\infty$, we must have $C\leq0$. Assume that $C<0$. Then, for $y$ sufficiently large, there exists $\delta>0$ such that
$ry'\geq\delta y^2$.
Arguing as above,
\[
\frac{d}{dy}\log r
=
\frac{1}{ry'}
\leq
\frac{1}{\delta y^2},
\]
so $r(y)$ remains bounded as $y\to+\infty$, again contradicting $r_{\mathrm{sup}}=+\infty$. Therefore $C=0$, and $g$ is scalar-flat.

Conversely, suppose that $g$ is cscK with non-negative scalar curvature. Then $D=0$ and $C\geq0$. By Lemma \ref{prelemmagood},
\[
y(r)\longrightarrow+\infty
\qquad\text{as }r\to r_{\mathrm{sup}}^-.
\]
Since $ry'(r)>0$, the equation
\[
ry'
=
y-\frac{A}{y^{n-1}}-\frac{B}{y^{n-2}}-Cy^2
\]
forces $C\leq0$. Hence $C=0$, so $g$ is scalar-flat. Thus
\[
ry'
=
y-\frac{A}{y^{n-1}}-\frac{B}{y^{n-2}}.
\]
Since $y$ is strictly increasing and tends to $+\infty$, it admits an inverse $r=r(y)$ for $y$ large, and
\[
\frac{dr}{dy}
=
\frac{r}{y-Ay^{1-n}-By^{2-n}}.
\]
Since $n\geq2$, there exist $Y_0>0$ and $C_0>0$ such that
\[
y-Ay^{1-n}-By^{2-n}\leq y+C_0
\]
for all $y\geq Y_0$. Hence
\[
\frac{1}{r}\frac{dr}{dy}
\geq
\frac{1}{y+C_0}.
\]
Integrating from $Y_0$ to $y$ gives
\[
\log\frac{r(y)}{r(Y_0)}
\geq
\log\frac{y+C_0}{Y_0+C_0},
\]
and therefore
\[
r(y)\geq r(Y_0)\frac{y+C_0}{Y_0+C_0}.
\]
Letting $y\to+\infty$, we obtain $r(y)\to+\infty$, and hence $r_{\mathrm{sup}}=+\infty$.

Therefore, in the infinitely projectively induced case,
$r_{\mathrm{sup}}=+\infty$ if and only if 
$g$ is scalar-flat.
Combining this with the finitely projectively induced case proves
$r_{\mathrm{sup}}=+\infty$ if and only if 
$g$ is cscK with non-negative scalar curvature.
\end{proof}

We next record a consequence of part \eqref{PartExtremal} of \Cref{thmNew} that will be used in the classification of the cscK case.
\begin{cor}\label{lemmagood}
Let $g$ be a radial cscK metric with non-negative scalar curvature and assume that $g$ is infinitely projectively induced. Then the function $y(r)=rf'(r)$ satisfies
\begin{equation}\label{ypsinew}
r\,y'(r)
=
y(r)
-\frac{A}{y(r)^{\,n-1}}
-\frac{B}{y(r)^{\,n-2}}
\end{equation}
for some $A,B\in\R$. Moreover, $y(r)=rf'(r)$ extends to a meromorphic function $y(\xi)$ on $\C$ which is holomorphic on a neighborhood of the non-negative real half-line.
\end{cor}

\begin{proof}
By part \eqref{PartExtremal} of \Cref{thmNew}, $g$ is scalar-flat and
$r_{\mathrm{sup}}=+\infty$. Hence Theorem \ref{lemmasimple} gives
$C=D=0$, and therefore
\[
r\,y'
=
y-\frac{A}{y^{n-1}}-\frac{B}{y^{n-2}},
\]
which is \eqref{ypsinew}.

Moreover, Theorem \ref{lemmaimp} gives
\[
F(r)=e^{f(r)}=\sum_{j=0}^{+\infty}a_jr^j
\]
with infinite radius of convergence, since $r_{\mathrm{sup}}=+\infty$.
Hence $F$ defines an entire function on $\C$, and
\[
y(\xi)=\xi\frac{F'(\xi)}{F(\xi)}
\]
is meromorphic on $\C$.

It remains to observe that $y$ is holomorphic in a neighborhood of the non-negative real half-line. Indeed, $F(r)>0$ for every $r>0$. At $r=0$, if $k$ is the first index such that $a_k>0$, then
\[
F(\xi)=\xi^kH(\xi),
\qquad
H(0)=a_k\neq0,
\]
and therefore
\[
\xi\frac{F'(\xi)}{F(\xi)}
=
k+\xi\frac{H'(\xi)}{H(\xi)},
\]
which is holomorphic near $\xi=0$.
\end{proof}

We will now piece together Corollary \ref{lem:y_expansion} and Corollary \ref{lemmagood} to compute the coefficients of the ODE \eqref{ypsinew}.

	\begin{lem}\label{lem:compute_AB}
    Let $g$ be a radial cscK metric with non-negative scalar curvature and assume that  $g$ is infinitely projectively induced. Then the function $y(r)$ associated to $g$ satisfies the ODE \eqref{ypsinew} with
		\[
		A=(l+1-n)\,k^n,\quad B=(n-l)\,k^{\,n-1}.
		\]
	\end{lem}
	
	\begin{proof}
		Using \eqref{eqexpansiony} in Corollary \ref{lem:y_expansion} write $y(r)=k+Kr^l+o(r^l)$, and define
		\[
		H(r)\;:=\;(y - r\,y')\,y^{\,n-1},
		\]
		so that the ODE \eqref{ypsinew}, that $y$ satisfies due to Corollary \ref{lemmagood}, is equivalent to
        \[
        H(r)=A+B\,y(r).
        \]
		
		Observe first that from the expansion $y(r)=k+Kr^l+o(r^l)$, it follows that \(y(r)\to k\) and \(r\,y'(r)\to0\) when \(r\to0\). Thus
		\[
		H(r)=(y - r\,y')\,y^{\,n-1}\;\longrightarrow\;k\cdot k^{\,n-1}=k^n.
		\]
		Since \(H(r)=A+B\,y(r)\to A+B\,k\), we get
		\begin{equation}\label{eqAB1}
			A + B\,k = k^n.
		\end{equation}
		
		Now observe that from \(H(r)=A+B\,y(r)\), we have
		\[
		B
		=\lim_{r\to0}\frac{H(r)-H(0)}{y(r)-y(0)}
		=\lim_{r\to0}\frac{H(r)-k^n}{y(r)-k},
		\]
		and since numerator and denominator vanish, by l'Hopital's rule
		\[
		B
		=\lim_{r\to0}\frac{H'(r)}{y'(r)}.
		\]
		A short computation yields
		\[
		\frac{H'(r)}{y'(r)}
		=(n-1)y^{n-1}
		-r\,\frac{y''}{y'}\,y^{n-1}
		-(n-1)r\,y'\,y^{n-2},
		\]
        and, since $r\,y'(r)=K\,l\,r^l+o(r^l)\to 0$ and 
		\[
		\frac{r\,y''(r)}{y'(r)}
		=\frac{K\,l(l-1)\,r^{l-1}+o(r^{l-1})}{K\,l\,r^{l-1}+o(r^{l-1})}
		\;\longrightarrow\;l-1,
		\]
		we have
		\begin{equation}\label{eqAB2}
			B=\lim_{r\to 0}\frac{H'(r)}{y'(r)}= (n-1)k^{n-1} - (l-1)k^{n-1}
			=(n-l)\,k^{\,n-1}.
		\end{equation}
	Finally, substituting \eqref{eqAB2} into \eqref{eqAB1} gives
		\[
		A = k^n - B\,k
		= k^n - (n-l)\,k^n
		= (l+1-n)\,k^n.
		\]
		This completes the proof.
	\end{proof}

	We conclude this section with the proof of part \eqref{PartcscK} of \Cref{thmNew}.
	
	\begin{proof}[Proof of \eqref{PartcscK}] 
		We shall prove the following more explicit form of its conclusion. Up to a linear change of coordinates, one of the following holds:
		\begin{enumerate}[(2a)]
			\item
			$g=m g_{FS}$ for some positive integer $m$, where $\omega_{FS}=\frac{i}{2}\partial\bar\partial\log(1+|z|^2)$;  \label{casoA}
			\item
			$g=g_0$, where $\omega_0=\frac{i}{2}\partial\bar\partial |z|^2$;  \label{casoB}
			\item
			$n=2$ and $g$ is a generalized Burns--Simanca metric $g_{k,BS}$ for some positive integer $k$, i.e. $
\omega_{k,BS}=\frac{i}{2}\partial\bar\partial\bigl(k\log|z|^2+|z|^2\bigr).
$ \label{casoC}
		\end{enumerate}

			If $g$ is finitely projectively induced then, by  \cite[Theorem 1.1]{LSZext}  $g$ is a complex space form and hence we are in case (2a) by Calabi's results.
		Hence assume $g$ to be infinitely projectively induced. Then, by the previous results, the function $y(r)=rf'(r)$ satisfies the ODE

				\begin{equation}\label{ODEgood}
					r\,y'(r)
					= y(r)
					\;-\;\frac{(l+1-n)\,k^n}{y(r)^{\,n-1}}
					\;-\;\frac{(n-l)\,k^{\,n-1}}{y(r)^{\,n-2}}
				\end{equation}
				for integers $k\geq 0$ and $l>0$.
		We will study three cases separately: $k=0$, $k\geq1$ with $n=2$ and  $k\geq1$ with $n\geq3$.
		
		If $k=0$, we deduce by \eqref{ODEgood} that $ry'(r)= y(r)$ yielding $y(r)=ar$, for some $a>0$ and it follows that we are in case (2b). 
		
		When $n=2$ and $k\geq 1$, equation \eqref{ODEgood}  becomes
		$$r\,yy'= y^2+(l-2)ky - (l-1)\,k^2=(y-k)(y+(l-1)k)$$
		and hence
		\[
		\left(\frac{1}{y - k} + \frac{l - 1}{y + (l - 1)k}\right)y' = \frac{l}{r}.
		\]
		Rewriting as
		\[
		\frac{d}{dr}\log(y-k)+\frac{d}{dr}\log(y+(l-1)k)^{\,l-1}
		=\frac{d}{dr}\log r^l,
		\]
		and integrating one gets
		\[
		(y(r)-k)\,\bigl(y(r)+(l-1)k\bigr)^{\,l-1}
		= a\,r^l,
		\]
		where $a=K(lk)^{l-1}>0$. We claim that $l=1$. Indeed, assume by contradiction $l>1$. Since the function
		$$
		R(y(\xi)):=(y(\xi)-k)\,\bigl(y(\xi)+(l-1)k\bigr)^{\,l-1},
		$$
		is a polynomial in $y$, and $y$ is meromorphic by Corollary \ref{lemmagood}, then $y$ cannot have finite poles, that would be poles of $R(y)$, whereas $a\xi^l$ is entire. This forces $y$ to be entire. Moreover, from $R(y(\xi))=a\xi^l$ it follows that $y$ has at most linear growth, that is $y$ is a polynomial of degree at most one.
On the other hand, near \(\xi=0\) we have
\[
y(\xi)=k+K\xi^l+o(\xi^l),
\qquad K>0.
\]
Since \(l>1\), this gives \(y'(0)=0\). A polynomial of degree at most one with
\(y'(0)=0\) is constant, contradicting the fact that \(R(y(\xi))=a\xi^l\) is non-constant.

			Thus $l=1$, i.e. $y=ar+k$. This leads to $f=\log r^k+ar$ and we fall in case (2c). 
		
		To complete the proof, we need to show that one cannot have both $n\geq 3$ and $k\geq 1$. Suppose
		the contrary. We distinguish two cases: $n=l,l+1$ and $n\neq l,l+1$. 
When \(n=l+1\) equation \eqref{ODEgood} becomes
\[
r y'(r)y(r)^{n-2}=y(r)^{n-1}-k^{n-1},
\]
thus
\[
\frac{d}{dr}\log\bigl(y(r)^{n-1}-k^{n-1}\bigr)
=
\frac{n-1}{r},
\]
and therefore
\[
y(r)^{n-1}-k^{n-1}=a r^{n-1}
\]
for $a=(n-1)k^{n-2}K>0$. By the meromorphic extension of \(y\), the identity extends to
\[
y(\xi)^{n-1}=k^{n-1}+a\xi^{n-1}
\]
on \(\mathds C\). Since  \(n\geq 3\), this contradicts the fact that the polynomial
\[
k^{n-1}+a\xi^{n-1}
\]
has \(n-1\) distinct non-zero zeros, all of order one, whereas the zeros of the
\((n-1)\)-st power of a meromorphic function must have order divisible by \(n-1\).

The case \(n=l\) is analogous. In that case we obtain
\[
y(\xi)^n=k^n+a\xi^n,
\]
with \(a=nk^{n-1}K>0\). The polynomial \(k^n+a\xi^n\) has \(n\) distinct non-zero
simple zeros, whereas the zeros of the \(n\)-th power of a meromorphic function must
have order divisible by \(n\), which is impossible.

		If $n\neq l , l+1$,  again by  \eqref{ODEgood}, we can write
		\begin{equation}\label{ypol}
			\xi\,y^{n-1}y'=P(y).
		\end{equation}
		where,
		$$P(y):=y^n -(n-l)k^{n-1}y+(n-l-1)\,k^n$$  
		is a polynomial in $y$.
		A very short computation shows that $P(y)$ and $P'(y)$ have no common roots so that the polynomial $P(y)$ has $n$ distinct roots $y_1=k, y_2, \dots , y_n$. Moreover $0$ is not a root of $P$, since $ P(0)=(n-l-1)k^n\neq 0$ for $n\neq l+1$.  Observe that $y$ cannot vanish, since $y(\xi_0)=0$ for some $\xi_0\in\mathds C$, implies $\xi_0\neq 0$, because $y(0)=k>0$, and evaluating \eqref{ypol} at $\xi=\xi_0$, the left hand side would be zero, while the right hand side is \(P(0)\neq 0\), a contradiction.
		
		Now either one of the following holds.
		\begin{enumerate}
			\item
			$y(\xi_0)=y_{j_0}$,  for some non-zero $\xi_0\in\C$ and some $j_0=2, \dots, n$;\label{case1}
			\item
			$y(\xi)\neq y_j$, for any $\xi$ and for any $j=2, \dots , n$.\label{case2}
		\end{enumerate}

		If \eqref{case1} holds, then, since \(\xi_0\neq 0\) and \(y_{j_0}\neq 0\), \eqref{ypol} can be
written near \((\xi_0,y_{j_0})\) as
\[
y'=\frac{P(y)}{\xi y^{n-1}},
\]
with holomorphic right hand side. Since \(P(y_{j_0})=0\), the constant function
\(y\equiv y_{j_0}\) is a local solution. By uniqueness for holomorphic ODE, \(y\) is constant in a neighbourhood of \(\xi_0\), and hence
constant by analytic continuation.
		This contradicts \(y(0)=k\), since \(y_{j_0}\neq k\). When \eqref{case2} holds, the meromorphic function \(y\) omits the three distinct values \(0,y_2,y_3\). By the Little Picard theorem, a non-constant meromorphic function on \(\mathds C\) can omit at most two values. Hence \(y\) must be constant, contradicting the expansion \(y(r)=k+Kr^l+o(r^l)\), with \(K>0\).
\end{proof}

    \section{Radial K\"ahler--Einstein metrics with negative scalar curvature}\label{sec:examples}
In this section we give the proof of \eqref{PartKE} of \Cref{thmNew}. Before doing so, we present two families of infinitely projectively induced radial cscK metrics with negative scalar curvature which do not have constant holomorphic sectional curvature. These examples also illustrate some of the difficulties involved in the classification of radial projectively induced cscK metrics with negative scalar curvature, which remains open in general.

  \begin{example}\label{ex:negative}\rm
      For $l\in \mathds Z ^+$, $l>0$, define:
  $$
f(r)=(l-1)\log(r)-2\log(1-r^l).
$$
The function $f$ is a K\"ahler potential for a radial cscK metric $g_l$ on the punctured unit ball of $\mathds C^2$, that satisfies the ODE \eqref{ypsi} with $C=-\frac 12$, $A=0$, $B=\frac{l^2-1}{2}$.
  The associated $y$ function reads:
  $$
  y(r)=rf'(r)=\frac{l-1+(l+1)r^l}{1-r^l}.
  $$
 When $l=1$, $g_1$ is the hyperbolic metric, while for $l=3$ the metric $\frac12g_3$ appeared in \cite[Example~3]{LSZext}. Since:
 $$
 e^{f(r)}=\frac{r^{l-1}}{(1-r^l)^2}=\sum_{j=1}^\infty jr^{lj-1},
 $$
  the metrics $g_l$ are all infinitely projectively induced. It is worth pointing out that, for \(c>0\), the rescaled metric \(cg_l\) is projectively
induced if and only if $c(l-1)$ is a non-negative integer.
Indeed,
\[
e^{c f(r)}=r^{c(l-1)}(1-r^l)^{-2c},
\]
and the binomial expansion of \((1-r^l)^{-2c}\) has positive coefficients for every
\(c>0\).
 \end{example}

\begin{example}\label{ex:negativeA}\rm
Let $k\in\mathds{Z}^{+}$ and $\lambda>0$, and define
\[
H(t):=
\frac{t+\sqrt{1+\frac{t^2}{2}}}
     {2\left(1-\frac{t^2}{2}\right)}.
\]
Consider the radial K\"ahler potential
\begin{equation}\label{newnegativepotential}
f_{k,\lambda}(r)
=
k\log r
+
k\int_{0}^{\lambda r^2}H(t)\,dt
\end{equation}
on the punctured ball
\[
0<r<\left(\frac{\sqrt{2}}{\lambda}\right)^{1/2}.
\]

Set $t=\lambda r^2$. A direct computation gives
\[
y(r)=rf_{k,\lambda}'(r)=ku(t),
\]
where
\[
u(t)
=
1+
\frac{t^2+t\sqrt{1+\frac{t^2}{2}}}
     {1-\frac{t^2}{2}}.
\]
The function $y$ satisfies
\[
ry'
=
y+\frac{k^2}{y}-3k+\frac{y^2}{k}.
\]
Therefore, in the notation of \eqref{ypsi},
\[
A=-k^2,\qquad B=3k,\qquad C=-\frac{1}{k},
\qquad D=0.
\]
It follows that $f_{k,\lambda}$ defines a radial cscK metric
$g_{k,\lambda}$ with scalar curvature
\[
\operatorname{scal}(g_{k,\lambda})=-\frac{6}{k}.
\]

In order to see that $g_{k,\lambda}$ is infinitely projectively
induced, set
\[
E_1(t)
=
\exp\left(\int_0^t H(\tau)\,d\tau\right).
\]
Then
\[
e^{f_{k,\lambda}(r)}
=
r^k E_1(\lambda r^2)^k.
\]
It is not difficult to verify that the Taylor expansion of $E_1$
at the origin has strictly positive coefficients. Since
$k\in\mathds{Z}^{+}$, the same is true for $E_1^k$, and hence
$g_{k,\lambda}$ is infinitely projectively induced.

Notice that this example belongs to the region $A<0$, whereas the
family $g_l$ above satisfies $A=0$. In particular, it is not obtained
from that family by a homothety. The parameter $\lambda$ can be
removed by a linear dilation of the coordinates.
\end{example}
 
We can now focus on the proof of \eqref{PartKE} of \Cref{thmNew} for which we need the following technical lemma.

\begin{lem}\label{tec}
Let $n\geq 2$, $s>0$, $k>0$, and set $L:=n+(n+1)s$. Define:
\begin{equation}\label{pu}
P(u):=u^n+su^{n+1}-(1+s),
\end{equation}
and let $t=t(u)$ be the germ at $u=1$ defined by
\begin{equation}\label{tu}
\frac{d\log t}{du}=\frac{L u^{n-1}}{P(u)},\quad t(u)=u-1+O((u-1)^2).
\end{equation}
Define also $Q(u)$ by $P(u)=(u-1)Q(u)$, i.e. 
\begin{equation}\label{ququ}
Q(u):=\sum_{j=0}^{n-1}u^j+s\sum_{j=0}^nu^j,
\end{equation}
and let $E=E(u)$ be the germ defined by:
\begin{equation}\label{logE}
\frac{d\log E}{du}=\frac{k u^{n-1}}{Q(u)},\quad E(1)=1.
\end{equation}
Let $E(t)$ be the corresponding germ at $t=0$, namely $E(t(u))=E(u)$. Then the Taylor expansion
\begin{equation}\label{Et}
E(t)=\sum_{j=0}^\infty e_jt^j
\end{equation}
cannot have all coefficients $e_j\geq 0$.
\end{lem}
  \begin{proof}
  We divide the proof in five steps.
  \begin{enumerate}
  \item[Step 1:] {\bf $t$ as a function of $u$}\\
  Observe that since $Q(u)>0$ for $u>0$, the only positive zero of $P(u)$ is $u=1$. Since $P(u)=(u-1)Q(u)$ and $Q(1)=L$, the function
$$
H(u)=\frac{L u^{n-1}}{P(u)}-\frac{1}{u-1}=\frac{1}{u-1}\left(\frac{L u^{n-1}}{Q(u)}-1\right),
$$
extends holomorphically to a neighborhood of $u=1$, since
the numerator in the last expression vanishes at $u=1$, because $Q(1)=L$, hence the apparent pole at $u=1$ is removable.
It follows that by \eqref{tu},
 $$
\frac{d\log t}{du}=\frac{1}{u-1}+H(u)=\frac{d\log(u-1)}{du}+H(u)
 $$
 that is
 $$
 \frac{d}{du}\log\left(\frac{ t(u)}{u-1}\right)=H(u),
 $$
which integrated gives:
 $$
 \log\left(\frac{ t(u)}{u-1}\right)=\int_1^uH(v)dv+{\rm constant}.
 $$
 Since $t(u)=u-1+O((u-1)^2)$, the constant is actually 0, and we get:
 $$
  t(u)=(u-1)e^{\int_1^uH(v)dv}
 $$
as a real analytic continuation on the positive real axis, away from $u = 1$. In particular, for $0<u<1$, we get $t(u)<0$ and hence $t_*:=t(0)<0$, while for $u>1$, $t(u)>0$, and since as $u\rightarrow +\infty$
$$
\frac{L u^{n-1}}{P(u)}\sim \frac{L u^{n-1}}{s u^{n+1}}=\frac{L}{su^2},
$$
it follows
$$
H(u)=-\frac1{u}+O\left(\frac1{u^2}\right).
$$
Thus, for some real constant $c_0$,
$$
\int_0^uH(v)dv=-\log u+c_0+O\left(\frac1u\right),
$$
and we get
$$
e^{\int_0^uH(v)dv}=\frac{e^{c_0}}{u}\left(1+O\left(\frac1u\right)\right),
$$
hence as $u\rightarrow +\infty$
$$
t(u)=(u-1)\frac{e^{c_0}}{u}\left(1+O\left(\frac1u\right)\right)\rightarrow e^{c_0},
$$
and we can define
$$
t_\infty:=\lim_{u\rightarrow +\infty} t(u)=e^{c_0}.
$$
  \item[Step 2:] {\bf $|t_*|<t_\infty$}\\
Set $x=\frac1u$, and define
$$
\Phi(x):=\log\left(t(1/x)\right)-\log(|t(x)|).
$$
Then:
$$
\lim_{x\rightarrow 0^+}\Phi(x)=\log \frac{t_\infty}{|t_*|},
$$
while
$$
\lim_{x\rightarrow 1^-}\Phi(x)=0.
$$
Thus:
$$
\log \frac{t_\infty}{|t_*|}=\Phi(0)-\Phi(1)=-\int_0^1\Phi'(x)dx.
$$
Let us show that $\Phi'(x)<0$ for $0<x<1$. Observe that by $P(u)=(u-1)Q(u)$ we get
$$
P(1/x)=\frac{1-x}xQ(1/x),
$$
and by \eqref{tu}
$$
\frac d{dx}t(1/x)=-\frac{t(1/x)L x^{-n-1}}{P(1/x)},\qquad \frac d{dx}|t(x)|=\frac{|t(x)|L x^{n-1}}{P(x)},
$$
and 
$$
\Phi'(x)=\frac{d}{dx}\log t(1/x)-\frac{d}{dx}\log |t(x)|=-\frac{L x^{-n-1}}{P(1/x)}-\frac{L x^{n-1}}{P(x)}.
$$
To conclude, let us show that $\Phi'(x)<0$, that is
$$
\frac{L x^{-n-1}}{P(1/x)}+\frac{L x^{n-1}}{P(x)}>0.
$$
Since for $0<x<1$, $P(1/x)>0$ while $P(x)<0$, it is equivalent to
$$
 x^{-n}P(x)+P(1/x) x^{n}<0,
$$
that by
$$
P(1/x) x^{n}=\frac{1-x}{x}\left(Q(x)+x^n-1\right),
$$
it reduces to
$$
(1-x)\left(\frac{1}{x}-\frac1{x^{n}}\right)Q(x)+\frac{(x^n-1)(1-x)}{x}<0,
$$
which is always true for $0<x<1$, $n\geq 2$, since $(1-x)Q(x)>0$, $x^{-n}>x^{-1}$, $(x^n-1)(1-x)<0$.
  \item[Step 3:] {\bf Local behavior at $u=0$}\\
 By \eqref{pu} and \eqref{tu} we get
$$
\frac{dt(u)}{du}=\frac{t(u)Lu^{n-1}}{P(u)}=-\frac{t_*L}{1+s}u^{n-1}+O(u^{2n-1})
$$
where we used that $t(u)=t_*+o(1)$. It follows that
$$
t(u)=t_*-\frac{t_*L}{n(1+s)}u^{n}+O(u^{2n})
$$
and thus
$$
1-\frac{t}{t_*}=\frac{L}{n(1+s)}u^{n}+O(u^{2n})
$$
which implies that the term $u^n$ is holomorphic as a function of $t$ near $t_*$,
since we can write
\begin{equation}\label{un}
u^n=\frac{n(1+s)}{L}\left(1-\frac{t}{t_*}\right)+O\left(\left(1-\frac{t}{t_*}\right)^2\right),
\end{equation}
while the term $u$ reads
$$
u=\left(\frac{n(1+s)}{L}\right)^{1/n}\left(1-\frac{t}{t_*}\right)^{1/n}\left(1+O\left(1-\frac{t}{t_*}\right)\right),
$$
and thus
\begin{equation}\label{un1}
u^{n+1}=\left(\frac{n(1+s)}{L}\right)^{1+1/n}\left(1-\frac{t}{t_*}\right)^{1+1/n}\left(1+O\left(1-\frac{t}{t_*}\right)\right).
\end{equation}

\item[Step 4:] {\bf $E$ as a function of $t$}\\
Since by \eqref{ququ} $Q(u)=(1+s)(1+u+O(u^2))$, we get
$$
\frac1{Q(u)}=\frac1{1+s}(1-u+O(u^2)),
$$
and by \eqref{logE}
$$
\frac{d\log E}{du}=\frac{k}{1+s}\left(u^{n-1}-u^{n}+O(u^{n+1})\right)
$$
that is
$$
\log E=\log E_*+\frac{k}{n(1+s)}u^n-\frac{k}{(n+1)(1+s)}u^{n+1}+O(u^{n+2}),
$$
where $E_*:=E(0)$ is nonzero, since by \eqref{logE},
$$
\log E(u)=\int_1^u\frac{k v^{n-1}}{Q(v)}dv
$$
and as $Q(u)>0$ for $u\in [0,1]$, the integral is finite in the interval $[0,1]$ and
$$
E(0)=e^{-\int_0^1\frac{k v^{n-1}}{Q(v)}dv}\neq 0.
$$ 
Hence
$$
E(u)=E_*+a_nu^n+a_{n+1}u^{n+1}+O(u^{n+2}), \qquad a_{n+1}=-\frac{kE_*}{(n+1)(1+s)}\neq 0.
$$ 
By \eqref{un} and \eqref{un1}, we get
\begin{equation}\label{Etexp}
E(t)=E_{\rm reg}(t)+A\left(1-\frac{t}{t_*}\right)^{1+1/n}+O\left(\left(1-\frac{t}{t_*}\right)^{1+2/n}\right),
\end{equation}
with 
$$
A=-\frac{kE_*}{(n+1)(1+s)}\left(\frac{n(1+s)}{L}\right)^{1+1/n}\neq 0
$$
and $E_{\rm reg}$ holomorphic near $t_*$.

\item[Step 5:] {\bf Conclusion}\\
Assume by contradiction that all the coefficients in \eqref{Et} are nonnegative. Since by \eqref{Etexp} $E(t)$ has a singularity at $t_*$, its radius of convergence $R$ is less than or equal to $|t_*|$. By Vivanti--Pringsheim Theorem (see e.g. \cite[p.~235]{remmert}), $E(t)$ must have a singularity at $R\leq |t_*|$, but this is not possible since for every $t<t_\infty$, $u(t)>1$ and $dt/du\neq 0$, hence $E(t)$ extends holomorphically to a neighborhood of each point of the interval $(0,t_\infty)$.
\end{enumerate}
  \end{proof}

    \begin{proof}[Proof of \eqref{PartKE} of \Cref{thmNew} ]
    As observed after the statement of \Cref{thmNew}, the cases with non-negative Einstein constant follow from part \eqref{PartcscK} of \Cref{thmNew}  itself. It remains only to treat the case $\lambda<0$. In this case, by Theorem \ref{lemmasimple}, the associated function $y(r)=rf'(r)$ satisfies the ODE:
\begin{equation}\label{altraode}
   ry'=y-\frac A{y^{n-1}}+|C|y^2.
\end{equation}
    Observe that by \cite[Theorem 1.1]{LSZext} if the ambient space is finite dimensional then a radial extremal K\"ahler metric has constant holomorphic sectional curvature. Thus, assume that $g$ is infinitely projectively induced.  By Corollary \ref{lem:y_expansion} we can write
\begin{equation}\label{altrayr}
    y(r)=k+K r^l+o(r^l), \quad k\in\mathds Z_{\geq 0},\ l\in\mathds Z_{>0}, \ K>0,
\end{equation}
    and by Theorem \ref{lemmaimp}, 
    $$
    F(r)=e^{f(r)}=\sum_{j=0}^\infty a_j r^j, \qquad a_j\geq 0,
    $$
    where $a_k$ is the first nonvanishing coefficient.
    
    If  $k=0$, then $y(r)\rightarrow 0$ as $r\rightarrow 0$, thus $A=0$ and by Theorem \ref{lemmasimple} the metric is a multiple of the complex hyperbolic metric, as wished. Thus, assume $k>0$. 
    The rest of the proof is by contradiction and will follow by Lemma \ref{tec}. We assume that $\omega$ as above is projectively induced with $k>0$.
    
    Taking the limit for $r\rightarrow 0$ in both sides of \eqref{altraode} we get $A=k^{n}(1+|C|k)$. Setting
    $$
    u(r):=\frac{y(r)}k,\quad s:=|C|k,
 $$
 \eqref{altraode} becomes
$$  
 ru'=u-(1+s)u^{1-n}+su^2
$$ that is, multiplying by $u^{n-1}$,
\begin{equation}\label{altraodeu}
 P(u):=  ru'u^{n-1}=u^n-(1+s)+su^{n+1},
\end{equation}
 with $P(1)=0$ and $P'(1)=n+(n+1)s=:L$. Observe that $L=l$, since by \eqref{altrayr} and \eqref{altraodeu}, we get
 $$
 l\frac{K}{k}r^{l}+o(r^l)=L\frac{K}{k}r^l+o(r^l),
 $$
 thus in particular $L\in \mathds Z^+$ and $t:=\frac{K}{k}r^L$ is a holomorphic radial variable. Then by \eqref{altraodeu}:
 $$
  d\log r = \frac{u^{n-1}}{P(u)}du, 
 $$
 and it follows that $\log t=\log \frac{K}{k}+L\log r$, and thus
\begin{equation}\label{logtrad}
 \frac{d\log t}{du}=\frac{Lu^{n-1}}{P(u)},\qquad t=u-1+O\left((u-1)^2\right).
\end{equation}
 Since $u$ is a holomorphic germ in $t$ near $r=0$, $u(r)=u(t)$ and we have
\begin{equation}\label{Fkk}
\frac{rF'(r)}{F(r)}=y(r)=k u(t)=k+k(u(t)-1) 
\end{equation}
 that is
\begin{equation}\label{Fkkk}
 r\frac d{dr}\log \left(\frac{F(r)}{r^k}\right)=k(u(t)-1).
\end{equation}
 Since 
 $$
 \frac{dt}{dr}=\frac{LK}{k}r^{L-1}=\frac{Lt}{r}
 $$
 implies
 $$
 Lt\frac{d}{dt}=r\frac{d}{dr},
 $$
 \eqref{Fkkk} becomes
\begin{equation}\label{Fkkk2}
 \frac{d}{dt}\log \left(\frac{F(r)}{r^k}\right)=\frac{k}L\frac{u(t)-1}t.
\end{equation}
 since $u(t)-1=t+O(t^2)$, $\frac{u(t)-1}t$ is holomorphic near $t=0$, and we can define
 $$
 E(t):=e^{\frac{k}L\int_0^t\frac{u(\tau)-1}\tau d\tau},
 $$
that satisfies $E(0)=1$ (notice that $t(1)=0$, so this is equivalent to $E(t(1))=1$).
 By \eqref{Fkkk2}
\begin{equation}\label{FEt}
F(r)=a_kr^k E(t),
\end{equation}
 where $a_k$ is the first nonvanishing coefficient in the expansion of $F(r)$. Observe that since $L$ is a positive integer, the coefficients in the Taylor expansion of $E(t)$  are non-negative. More precisely,
 $$
 \frac{F(r)}{a_kr^k}= E\left(\frac{K}{k}r^L\right),
 $$
 thus all the $a_j$'s with $j-k$ that is not an integer multiple of $L$ vanish, and
 $$
 E(t)=\sum_{m=0}^\infty\frac{a_{k+mL}k^m}{a_kK^m}t^m.
 $$
 To conclude, it remains to show that $E(t)$ satisfies the hypothesis of Lemma \ref{tec}, that implies that at least one of the coefficients in the Taylor expansion of $E(t)$ is negative, giving the contradiction. Combining \eqref{Fkk}, \eqref{Fkkk2}, and \eqref{FEt}, we get
 $$
 y=ku=k+ Lt\frac{ E'(t)}{E(t)}
 $$
 and by \eqref{logtrad} this is equivalent to
 $$
 \frac{d\log E}{du}=Lt\frac{ E'(t)}{E(t)}\frac{u^{n-1}
}{P(u)}=\frac{ku^{n-1}}{Q(u)},\qquad Q(u):=\frac{P(u)}{u-1},
 $$
 concluding the proof of the theorem.
    \end{proof}

\end{document}